\documentclass[12pt]{amsart}
\usepackage[margin=3cm]{geometry}
\usepackage{amsmath,amssymb,mathrsfs,latexsym,amsthm,xcolor,array}

\definecolor{gray}{rgb}{0.4,0.4,0.4}
\definecolor{white}{rgb}{1,1,1}
\definecolor{red}{rgb}{0.8,0,0}
\usepackage[T1]{fontenc}
\usepackage{tgtermes}
\usepackage{mathptmx}  
\usepackage[scaled=.92]{helvet}
\usepackage{amsthm,amsmath,amssymb}
\usepackage{mathrsfs}
\usepackage[numbers]{natbib} 
\usepackage[colorlinks,citecolor=blue,urlcolor=blue]{hyperref}  
\usepackage{enumerate}

\usepackage{latexsym,color,txfonts,array,graphicx}
\usepackage[all]{xy}
\usepackage{exscale}

\theoremstyle{definition}

\theoremstyle{remark}

\title{\emph{Teachers, Learners and Oracles}}

\author{Achilles A. Beros}
\address[Achilles A. Beros]{Department of Mathematics\\
University of Hawai`i at M{\=a}noa\\
Honolulu, HI 96822, USA}
\email{beros@math.hawaii.edu}

\author{Colin de la Higuera}
\address[Achilles A. Beros]{Laboratoire d'Informatique de Nantes Atlantique\\
	Universit\'e de Nantes\\
	2 rue de la Houssini\`{e}re BP 92208\\
	44322 Nantes Cedex 03\\
	FRANCE}
\email{cdlh@univ-nantes.fr}

\theoremstyle{plain}

\newtheorem{Theorem}{Theorem}[section]
\newtheorem{Proposition}[Theorem]{Proposition}
\newtheorem{Lemma}[Theorem]{Lemma}

\theoremstyle{definition}
\newtheorem{Definition}[Theorem]{Definition}

\newtheorem{Example}[Theorem]{Example}

\newcommand{\bbn}{\mathbb N}
\newcommand{\bbz}{\mathbb Z}

\newcommand{\content}[1]{\mbox{content}(#1)}

\newcommand{\desc}{\textsc{descriptors}}
\newcommand{\signed}{\textsc{signedInt}}

\newcommand{\upto} {\upharpoonright}

\newcommand{\seq}[2]{\{#1_{#2}\}_{#2 \in \bbn}}

\newcommand{\mi}[2]{\textsc{mi}_{#1}(#2)}

\begin{document}

\begin{abstract}
We exhibit a family of computably enumerable sets which can be learned within polynomial resource bounds given access only to a teacher, but which requires exponential resources to be learned given access only to a membership oracle.  In general, we compare the families that can be learned with and without teachers and oracles for four measures of efficient learning.
\end{abstract}

\maketitle

\section{Introduction}

In this paper, we address the question of whether or not the presence of a teacher as a computational aide improves learning.  A teacher is a computable machine that receives data and selects a subset of the data.  In the models we consider, a teacher receives an enumeration for a target and passes its data selection to the learner -- the learner does not have access to the original data.  The first natural question is if there are families that are learnable with a teacher, but not learnable without.  As will be obvious from the definitions presented in the next section, the answer is no: the learner can always perform an internal simulation of the learner-teacher interaction and output the result.  The second question is whether a teacher can improve efficiency.  For teacher models of learning, only the computational activity of the learner counts against the efficiency bound; the computational activity of the teacher is not counted.  Heuristically, the question is whether there is benefit to pre-processing data.  We will prove there can be an exponential improvement in efficiency.  In fact, there are situations where access to a teacher is better than access to a membership oracle about the target.

Various forms of and questions related to teaching have arisen in learning theory over the last few decades.  Work on the complexity of teaching families has given rise to the classical teaching dimension \cite{goldman-kearns} and more recently the recursive teaching dimension \cite{zilles-1, zilles-2}.  In \cite{zilles-2}, Zilles et al.~establish deep and interesting connections between recursive teaching dimension, Vapnik-Chervonenkis dimension and sample compression schemes (see \cite{floyd-warmuth} for more about sample compression).  Query learning has been a central topic in learning theory for even longer than teaching.  Numerous papers have been written both on the abilities of machines equipped with oracles to learn \cite{sammut-banerji, angluin-q1, angluin-q2} and on the properties of oracles that allow learning of certain target families \cite{stephan-noisy, jain-sharma, fortnow-etal, stephan-kummer}.

We add to the body of research on teaching and query learning by comparing the efficiency of the two learning modes.


\section{Background}

We will examine variants of Gold-style text learning of effectively describable sets of natural numbers.  In particular, the target objects will be \emph{computably enumerable sets}.

\begin{Definition}
A set, $S$, is \emph{computably enumerable (c.e.)} if there is a partial computable function, $f$, such that $S = \mathrm{dom}(f)$.  A sequence of sets, $\{A_n \}_{n \in \mathbb N}$ is called \emph{uniformly computably enumerable (u.c.e.)} if the set $\{ \langle a,i \rangle : a \in A_i \}$ is c.e.  We also call u.c.e.~sequences of sets \emph{indexed families} and call $n$ an index for $A_n$.  Note that in an indexed family a set may have multiple indices if the sequence $\{A_n \}_{n \in \mathbb N}$ has multiple instances of the same set.  For notational convenience, we regard indexed familes both as sequences and as sets and write $A \in \mathcal A$ meaning $(\exists n) (A = A_n)$.
\end{Definition}

We now remind the reader of some standard notation and concepts as well as introducing some notation specific to this paper.

\begin{enumerate}
\item $\phi$ denotes an acceptable universal Turing machine and hence, a partial computable function.  $\phi_{e,s}(x)$ is the state or value of the function described by the program coded by $e \in \mathbb N$ after $s$ computation stages on input $x$.  If the program execution has terminated, we write $\phi_{e,s}(x)\downarrow$, otherwise we write $\phi_{e,s}(x)\uparrow$.

\item $W_e$ is the c.e.~set coded by the program $e$ as the domain of $\phi_e$.  $\{ W_e \}_{e \in \mathbb N}$ is a u.c.e.~sequence of sets and enumerates all the c.e.~sets.  We write $\mathcal E$ for the set of all c.e.~sets.

\item For $n \in \bbn$, $\langle x_0,x_1,\ldots , x_n \rangle : \bbn^{n+1} \rightarrow \bbn$ is a polynomial-time computable encoding function such that $x_i \leq \langle x_0,x_1,\ldots , x_n \rangle$ for all $i \leq n$.  We also define a polynomial-time computable decoding function $(x)_n: \bbn \rightarrow \bbn^{n}$ which is the inverse function of encoding function $\langle x_0,x_1,\ldots , x_{n-1} \rangle$.  We define $A \otimes B = \{ \langle a,b \rangle : (a \in A) \wedge (b \in B) \}$.  We use $\otimes$ to partition $\bbn$ into an infinite number of infinite computable sets, $\bbn \otimes \{0\}, \bbn \otimes \{1\}, \ldots$.  Sets of this form are known as \emph{columns}, whereby $\bbn \otimes \{i \}$ is the $i^{th}$ column of $\bbn$.  As a shorthand, we will represent the $i^{th}$-column of $\bbn$ with the symbol $C_i$ and the $i^{th}$-column of $A \subseteq \bbn$ by $C_i(A)$.  Associated with $C_i$, we define $c_i$ to be a computable function such that $W_{c_i(x)} = W_x \cap C_i$.

\item We write $(x_0,x_1, \ldots , x_n)$ to denote the ordered tuple of elements (as opposed to the encoding of the ordered tuple, $\langle x_0,x_1,\ldots , x_n \rangle$).

\item We fix an encoding of polynomials as natural numbers and write $p^*$ to denote the encoding of a polynomial $p$.  The encoding is polynomial-time computable, as is the decoding, and maps onto $\bbn$.

\item $\signed : \bbn \rightarrow \bbz$ is the computable bijection such that $\signed(2n) = n$ and $\signed(2n+1) = -(n+1)$.

\item If $a$ is a string or natural number, then $a^i$ denotes the string which consists of $a$ repeated $i$ times.

\item For function composition we use the notation $f \circ g$ where $(f \circ g)(x) = f(g(x))$.

\item If $\sigma = a_0 \cdots a_n$ is a string, then $|\sigma| = n+1$ is the length of the string, $\sigma(k) = a_k$ and $\content{\sigma} = \{ \sigma(k) : k < |\sigma| \}$.

\item An enumeration of a non-empty set $A$ is an infinite sequence of elements of $A$ such that every element of $A$ appears in the sequence at least once.  We regard an enumeration as a stream of bits with markers between individual elements.  We will restrict our attention to non-empty sets.  Consequently, we need not consider enumerations of the empty set.

\item A learning machine (or learner) is a partial computable function that receives a string as input, may have access to oracle queries and outputs a natural number that is interpreted as a code for a set.  The outputs are called hypotheses and the sequence of hypotheses produced by a learner on initial segments of an enumeration is called the hypothesis stream.  When measuring efficiency, we allow a learner to skip an element of an enumeration for some fixed computational cost.

\item Given an interval $[0, n]$, where $n$ is unknown, but bounded by $a^m$, $n$ can be determined with $(m+1)^{a+1}$ or fewer oracle queries using the following algorithm.  First, determine the least $k_0$ such that $a^{k_0+1} \not\in [0, n]$.  We will obtain $k_0$ after at most $m+1$ queries.  Next, we repeat the process to determine the least $k_1$ such that $a^{k_0} + a^{k_1+1} \not\in [a^{k_0}, n]$.  By iterating this process at most $a+1$ times we find $n$.  We call this an exponential query search algorithm.
\end{enumerate}

%
%
%
%

We will consider learning models using combinations of three different data sources: enumeration, oracle and teacher.  All of the models we consider are forms of TxtEx-learning, or learning in the limit.  We begin with the definition of this fundamental learning model.

\begin{Definition}
Let $M$ be a computable learning machine, $\mathcal F = \{F_n \}_{n \in \mathbb N}$ an indexed family and $\{ a_n \}_{n \in \mathbb N}$ an enumeration (text) of a set $F \in \mathcal F$.
\begin{enumerate}
\item $M$ TxtEx-identifies $\{ a_n \}_{n \in \mathbb N}$ if 
\[
(\exists i)(\forall j) \big( M(a_0 \ldots a_{i+j}) = M(a_0 \ldots a_i) \wedge F_{M(a_0 \ldots a_i)} = F \big)
\]
If only the first condition above is met, i.e., $(\exists i)(\forall j)(M(a_0...a_{j+1}) = M(a_0...a_i))$, then we say that $M$ has \emph{converged} on the enumeration $\{a_n\}_{n \in \mathbb N}$.

\item $M$ TxtEx-learns $F$ if $M$ TxtEx-identifies every enumeration of $F$.

\item $M$ TxtEx-learns $\mathcal F$ if $M$ TxtEx-learns every $F \in \mathcal F$.
\end{enumerate}
\end{Definition}

All of the models we examine in this paper are variants of TxtEx-learning.  The parameters we will vary are linked to sources of information and the measurement of efficiency.  We state definitions of these variants starting from an arbitrary learning model.

\begin{Definition}\label{def-teacher-oracle}
Let L-learning be an arbitrary learning model.  
\begin{enumerate}
\item We say that $\mathcal F$ is \emph{L-learnable with a membership oracle} (denoted L[O]-learnable) if there is a learning machine, $M$, that L-learns $\mathcal F$ and has access to a membership oracle for the target it is learning.  As membership oracles are the only oracles we will consider, we often simply refer to a membership oracle as an oracle.

\item A function $T: 2^{<\bbn} \rightarrow 2^{<\bbn}$ is a \emph{teacher} if it is a computable function, $T(\sigma)$ is a prefix of $T(\tau)$ whenever $\sigma$ is a prefix of $\tau$, and $\content{T(\sigma)} \subseteq \content{\sigma}$.  We say that $\mathcal F$ is \emph{L-learnable with a teacher} (denoted L[T]-learnable) if there is a learner-teacher pair $(M,T)$ such that $M$ L-identifies every enumeration of the form $T \circ f$, where $f$ enumerates a member of $\mathcal F$.  

\item We say that $\mathcal F$ is \emph{L-learnable with a teacher and a membership oracle} (denoted L[T,O]-learnable) if there is a learner-teacher pair, $(M,T)$, such that $M$ has access to a membership oracle, $T$ has access to the query responses $M$ receives, and $M$ L-identifies every enumeration of the form $T \circ f$, where $f$ enumerates a member of $\mathcal F$.

%
\end{enumerate}
\end{Definition}

As is clear from the definition, the teacher serves to pre-process the text input before passing the elements deemed important to the learner.  In the subsequent sections, we will consider the different combinations of teacher and oracle with certain variants of TxtEx-learning.

When defining efficiency notions for learning, the first natural notion is that of polynomial run-time: the learner must converge within $p(e)$ computation steps, where $p$ is a polynomial and $e$ is a code for the target.  There are two problems with this definition.  First, apart from trivial cases, any learning process can be delayed arbitrarily by using an enumeration that repeats a single element of the target set.  Second, if a learning machine has produced an encoding of the target, but has failed to do so in polynomial run-time, a suitably larger and equivalent encoding can be chosen instead so that the run-time is appropriately bounded.  As we are considering indexed families, rather than general classes of c.e.~sets, we can address the second problem by fixing a reference index for every set in the family against which efficiency is measured.

\begin{Definition}
Let $\mathcal A = \{ A_n \}_{n \in \mathbb N}$ be an indexed family.  We define the \emph{minimal index of $A \in \mathcal A$} (symbolically, $\mi{\mathcal A}{A}$ to be the least $n$ such that $A_n = A$.
\end{Definition}

By restricting our attention to indexed families, we have a well-defined concept of polynomial bounds in the size of the target that is independent of the underlying numbering of the c.e.~sets, thereby addressing the second problem.  In the absence of an oracle or teacher the first problem remains.  Nevertheless, we include polynomial run-time among the notions of efficiency that we define below as it is reasonable when an oracle or teacher is present. 

We will address four measures of learning efficiency: Polynomial run-time, polynomial size dataset, polynomial size characteristic sample, and polynomial mind-changes.

We have also proved~\cite{beros-teachers-arxiv} the results presented in this paper for general classes of c.e.~sets equipped with an indexing function.  Nevertheless, in this paper we restrict our attention to the limited case of indexed families as it is a more familiar context than the indexed target families required by the general case.  In that more general case, the indexing function selects a unique code from the underlying numbering for each set in the class.  The codes output by the indexing function are taken as the reference against which efficiency is computed.

\section{Polynomial Run-Time}\label{prt-sec}

\begin{Definition}\label{prt-def}
An indexed family $\mathcal F = \seq{F}{n}$ is \emph{polynomial run-time learnable} (\emph{PRT-learnable}) if there is a machine $M$ and a polynomial $p$ such that for every enumeration $f$ of $F \in \mathcal F$, the learner $M$ converges to a correct index on $f$ in fewer than $p(\mi{\mathcal F}{F})$ computation steps.  If an oracle is accessed, oracle use must also be bounded by $p(\mi{\mathcal F}{F})$.  We use to PRT to denote the set of all PRT-learnable indexed families.
\end{Definition}

We will apply Definition \ref{def-teacher-oracle} to Definition \ref{prt-def} to obtain, for example, PRT[T]-learning and PRT[T], the PRT[T]-learnable indexed families.

Proposition \ref{prt-sans-to} demonstrates that PRT-learnability is much too restrictive in the absence of an oracle or teacher. 

\begin{Proposition}\label{prt-sans-to}
Let $\mathcal F$ be an indexed family.  If there are $A, B \in \mathcal F$ such that $A \neq B$ and $A \cap B \neq \emptyset$, then $\mathcal F$ is not PRT-learnable.
\end{Proposition}

\begin{proof}
Let $A$, $B$ and $\mathcal F$ be as in the statement, let $M$ be an arbitrary learning machine and $p$ an arbitrary increasing polynomial.  Also, let $a=\mi{\mathcal F}{A}$, $b=\mi{\mathcal F}{B}$ and let $x \in A \cap B$.  Define $f_A$ to be an enumeration of $A$ that begins with $x^{p(a) + p(b)}$ and $f_B$ be an enumeration of $B$ that begins with $x^{p(a) + p(b)}$.  If $M$ PRT-identifies $f_A$, then $M(\sigma)$ must be a code for $A$ for any $\sigma = x^{p(a)+i}$ for $i \geq 0$.  Similarly, if $M$ PRT-identifies $f_B$, then $M(\sigma)$ must be a code for $B$ for any $\sigma = x^{p(b)+i}$ for $i \geq 0$.  Thus, no machine can PRT identify both $f_A$ and $f_B$ and $\mathcal F$ is not PRT-learnable.

\end{proof}

On the other hand, there are many non-trivial indexed families which are PRT[O]-, PRT[T]- or PRT[T,O]-learnable.

\begin{Example}\label{prt-o-ex}
Define $F_n = [n, \infty)$, $G_{\langle m,n \rangle} = [m, n]$ and $H^k_n = \content{(n)_k}$.  The indexed families $\mathcal F = \seq{F}{n}$, $\mathcal G = \seq{G}{n}$ and $\mathcal H_k = \seq{H^k}{n}$ for $k\in \bbn$ are PRT[O]-learnable.
\end{Example}

\begin{Example}\label{prt-t-ex}
The indexed families in Example \ref{prt-o-ex} are also PRT[T]-learnable.  For example, consider $\mathcal H_k$ for some fixed $k$.  Define a teacher $T$ such that $T(a_0 \cdots a_n) = T(a_0 \cdots a_{n-1}) a_n$ if $a_n \not\in \{a_0, \ldots, a_{n-1}\}$ and outputs $T(a_0 \cdots a_{n-1})$ otherwise.  Define a learner $M$ that waits until it has received $k+1$ distinct numbers, $\{b_0, \ldots b_k\}$, from $T$ and then outputs $\langle b_0, \ldots b_k \rangle$.  $(M, T)$ PRT[T]-learns $\mathcal H_k$.
\end{Example}

\begin{Proposition}\label{prt-t-o}
$PRT \subset PRT[O] \subseteq PRT[T,O]$ and $PRT \subset PRT[T] \subseteq PRT[T,O]$.
\end{Proposition}

\begin{proof}
Let $\mathcal H_2$ be as above.  As observed in Examples \ref{prt-o-ex} and \ref{prt-t-ex}, $\mathcal H_2$ is PRT[O]-learnable and PRT[T]-learnable, but by Proposition \ref{prt-sans-to}, $\mathcal H_2$ is not PRT-learnable.  Thus, $PRT \subset PRT[O] \cap PRT[T]$.  The other containments follow from the definitions.

\end{proof}

We now produce indexed families that distinguish PRT[T]-learning from PRT[O]-learning and PRT[T,O]-learning from both PRT[O]- and PRT[T]-learning.  In order to prove that all of these distinctions are non-trivial, we introduce the concept of marked self-description.

\subsection{Marked Self-Describing Sets}

Including self-description in an object is an encoding technique on which many important learning theory examples are based.  Examples of self-description include the self-describing sets $\mathcal {SD} = \{ A\in \mathcal E : W_{min(A)} = A \}$, and the almost self-describing functions $\mathcal {ASD} = \{ f : \phi_{f(0)} =^* f \}$.  Many variants on the self-description theme have been explored in learning theory and inductive inference.

Our interest is in families that use carefully engineered self-description to calibrate the difficulty in identifying their members.  We will construct families whose members are not only self-describing, but also have their self-describing elements marked for ease of identification.  We say that such families exhibit \emph{marked self-description}.  In particular, we will use encapsulating objects that we call descriptors.

\begin{Definition}
For finite $X \subset \bbn$, a \emph{descriptor on the $i^{th}$-column} is a finite set $D = \{ \langle x,c_x,1, i \rangle : x \in X\}  \subseteq C_i(C_1)$ such that
\begin{enumerate}
\item $\sum_{x \in X} \signed(c_x) = 0$

\item $(\forall X' \subset X)\big(\sum_{x \in X'} \signed(c_x) \neq 0 \big)$

\item $\sum_{x \in X'} \signed(x) \geq 0$.
\end{enumerate}

Such a descriptor is said to \emph{describe} the natural number $n = \sum_{x \in X}\signed(x)$.  For $\langle x,c_x,1, i \rangle \in D$, we call $c_x$ the \emph{completion index} of the element.  For $n\in \bbn$, we define $\desc_i(n)$ to be the set of all descriptors on the $i^{th}$-column that describe $n$.
\end{Definition}

A descriptor can be thought of as a stream of data that includes parity bits to check the integrity of the data stream and where the intended message is the number described by the descriptor.  Thus, a machine can decide not only which elements are pieces of the descriptor (packets in the stream), but also decide when the entire descriptor has appeared in the enumeration (all the packets have been received).  By using a descriptor to encode the self-description for a set, we make the self-description instantly recognizable upon appearance in the enumeration.  For this reason, learning such a self-describing set can be achieved with no mind-changes.  In contrast to the degree to which we have made learning easier, we have potentially made \emph{efficient} learning harder.  By distributing the self-description into a large descriptor, we will create a scenario in which a very large amount of data is required to reach a correct decision.  We now proceed to our first result using these tools.

\begin{Lemma}\label{prt-t-not-o}
There is an indexed family $\seq{F}{n}$, where $F_n$ describes $n$, which is PRT[T]-learnable, but not PRT[O]-learnable.  We call this indexed family the \emph{marked self-describing sets} and designate it by $\mathcal{MSD}$.
\end{Lemma}

\begin{proof}
Fix $n \in \bbn$ and let learning machine $M$ and polynomial $p$ be such that $n = \langle m, p^*, i \rangle$, where $\phi_m = M$ and $i \in \{0,1\}$.  Without loss of generality, we may assume that $p$ is increasing.  Consider the situation where $M$ has access to the membership oracle for the singleton $\{ \langle 0,1,1,0 \rangle \}$ and define a computable function $q$ such that, for $\ell \in \mathbb N$, $q(\ell)$ is the greatest number about which $M$ queries the oracle when it receives inputs which are substrings of $\langle 0,1,1,0 \rangle^{\ell}$.  Note that $q$ is an increasing function.  Define $F_n$ to be a member of $\desc_0(n)$ such that
\begin{align}
F_n \cap [0, q(p(\langle m, p^*, 1 \rangle))] = \{ \langle 0,1,1,0 \rangle \}\label{msd-formula}
\end{align}
and chosen according to a fixed algorithm so that $\mathcal{MSD} = \seq{F}{n}$ is u.c.e.  

First, we show that $\mathcal{MSD}$ is PRT[T]-learnable.  We define a teacher $T$ as follows.  If $\content{\sigma}$ is not a descriptor, $T(\sigma)$ is the empty string.  If $D = \content{\sigma}$ describes $n$, then $T(\sigma) = \min(D)^i$ if $|\sigma| = |\sigma_0| + i$ where $\sigma_0$ is the shortest intial segment of $\sigma$ whose content contains $D$ and $i < n$; if $i = n$ then $T(\sigma) = \min(D)^n$.  Having output $\min(D)$ $n$ times, $T$ proceeds by enumerating $D$ in decreasing order.  Let $M$ be a machine that reads the output of $T$ and returns the number of elements in the output of $T$.  The teacher-learner pair learns $\mathcal{MSD}$ and the run-time of the learner is linear in the index of the target.

We now show that $\mathcal{MSD}$ is not PRT[O]-learnable.  To prove that $\mathcal{MSD}$ is not PRT[O]-learnable, fix a learner $M = \phi_m$, an increasing polynomial $p$ encoded by $p^*$, $n_0 = \langle m, p^*, 0 \rangle$ and $n_1 = \langle m, p^*, 1 \rangle$.  If $M$ PRT[O]-learns $\mathcal{MSD}$ with polynomial bound $p$, then it must succeed at identifying $F_{n_0}$ and $F_{n_1}$ within $p(n_1) \geq p(n_0)$ computation stages.  Choose $T_0$ and $T_1$ to be any enumerations of $F_{n_0}$ and $F_{n_1}$, respectively, which have $\langle 0,1,1,0 \rangle^{p(n_1)}$ as an initial segment.  When trying to identify $T_0$ and $T_1$, the learner must reach its final hypothesis before finding any elements of the target sets, $F_{n_0}$ and $F_{n_1}$, other than $\langle 0,1,1,0 \rangle$.  Whatever hypothesis $M$ converges to before completing the $p(n_1)$ length initial segment of either enumeration cannot code both sets.  Thus, $M$ fails to learn at least one of the two sets.  Since $M$ and $p$ were chosen arbitrarily we conclude that $\mathcal{MSD}$ is not PRT[O]-learnable.

\end{proof}

\begin{Lemma}\label{not-prt-t}
If $\mathcal F = \seq{F}{n}$ is an indexed family, $p$ a polynomial and there are indices $a, b_0, b_1, \ldots b_{p(a)-1}$ such that $F_{b_0} \subset F_{b_1} \subset \ldots \subset F_{p(a)-1} \subset F_a$, then $\mathcal F$ is not PRT[T]-learnable with polynomial bound $p$.
\end{Lemma}

\begin{proof}
Let $\mathcal F$, $p$, $a$ and $b_0, \ldots, b_{p(a)-1}$ be as in the statement and let $(M, T)$ be an arbitrary learner-teacher pair.  Let $\sigma_0$ be an initial segment of an enumeration of $F_{b_0}$ on which $M \circ T$ outputs an index for $F_{b_0}$ (if no such $\sigma_0$ exists, then $(M, T)$ has already failed to learn $\mathcal F$).  Given $\sigma_n$, an initial segment of an enumeration of $F_{b_n}$ for $n < p(a)-1$, define $\sigma_{n+1}$ to be an initial segment of an enumeration of $F_{b_{n+1}}$ extending $\sigma_n$ on which $(M, T)$ outputs an index for $F_{b_{n+1}}$.  Again, if no such extension can be found, then $M$ has failed to learn $\mathcal F$.  Let $T$ be an enumeration of $F_a$ which has $\sigma_{p(a)-1}$ as an initial segment.  Since $M$ changes hypothesis at least $p(a)$ times on $\sigma_{p(a)-1}$, either $M$ fails to identify $T$ or the runtime of the learner cannot be bounded by $p(a)$.

\end{proof}

\begin{Lemma}\label{prt-o-not-t}
There is a PRT[O]-learnable indexed family that is not PRT[T]-learnable.  We call this indexed family the \emph{column self-describing sets} and designate it by $\mathcal{CSD}$.
\end{Lemma}

\begin{proof}
Define 
\begin{align}
a_n = n + 1 + \sum_{i = 0}^{n-1} p_i(a_i),\label{csd-formula-1}
\end{align}
where $p_i$ is the polynomial such that $p_i^* = i$.  Fix $n \in \bbn$ and define $A_n = [0, a_n]\otimes\{p_n(a_n)\} \allowbreak\cup \bigcup_{i < p(a_n)} [0, a_n+i]\otimes\{i\}$ and $B_{n,i} = \bigcup_{j \leq i} [0,a_n+j]\otimes\{j\}$, for $i < p(a_n)$.  Finally, define $F_n = A_i$ if $n = a_i$ and $F_n = B_{i,j}$ if $n = a_i + j$ where $j < p_i(a_i)$.  Let $\mathcal{CSD} = \seq{F}{n}$.  To PRT[O]-learn $\mathcal{CSD}$, define $M$ to be a learning machine that uses the exponential query search algorithm to find the highest index non-empty column, queries about the members of the column, in increasing order, until the greatest element is found, and returns the value of this element.  Since the number of queries involved is polynomially bounded in $e$, $M$ witnesses the desired learnability.

Since $B_{n,0} \subset B_{n,1} \subset \cdots \subset B_{n,p(a_n)-1} \subset A_n$, for each polynomial, $p$, there is a subfamily of $\mathcal F$ that cannot be PRT[T]-learned with efficiency bound $p$.  Thus, $\mathcal F$ is not PRT[T]-learnable.

\end{proof}

Finally, we wish to distinguish PRT[T,O]-learning from both PRT[T]-learning and PRT[O]-learning.

\begin{Lemma}\label{prt-to-not-t-o}
There is an indexed family which is PRT[T,O]-learnable, but neither PRT[T]-learnable nor PRT[O]-learnable.
\end{Lemma}

\begin{proof}
To prove the claim, we must combine the strategies used in the proofs of Lemma \ref{prt-t-not-o} and Lemma \ref{prt-o-not-t}.  Define $F_n$ exactly as the members of $\mathcal{MSD}$ are defined except we modify formula (\ref{msd-formula}) to be
\[
F_n \cap [0, q(p(3\langle m, p^*, 1 \rangle))] = \{ \langle 0,1,1,0 \rangle \}.
\]
We also define $G_n$ exactly as the members of $\mathcal{CSD}$ are defined except that we replace formula (\ref{csd-formula-1}) by
\[
a_n = 3(n + 1) + \sum_{i = 0}^{n-1} p_i(3a_i)
\]
Finally, we define $\mathcal H = \seq{H}{n}$ where
\[
H_n = \begin{cases}
G_i &\mbox{if } n = 2i\\
F_i &\mbox{if } n = 2i+1
\end{cases}.
\]
We will show that $\mathcal H = \seq{H}{n}$ is PRT[T,O]-learnable, but neither PRT[T]-learnable nor PRT[O]-learnable.  To PRT[T,O]-learn $\mathcal H$, let $M$ be a learner which first determines if the target set contains 0 using an oracle query.  If the target does, then $M$ proceeds as the PRT[O]-learner in the proof of Lemma \ref{prt-o-not-t}, multiplying the hypotheses output by that learner by 2.  If the target does not contain 0, then $M$ proceeds as the PRT[T]-learner in the proof of Lemma \ref{prt-t-not-o}, multiplying the hypotheses output by that learner by 2 and adding 1.  $M$ PRT[T,O]-learns $\mathcal H$ with only a linear decrease in efficiency compared to the two learners from the previous Lemmas.

To see that $\mathcal H$ is neither PRT[O]-learnable nor PRT[T]-learnable, observe that the proofs of Lemmas \ref{prt-t-not-o} and \ref{prt-o-not-t} suffice to show that $\mathcal H$ contains two indexed subfamilies, one of which fails to be PRT[O]-learnable and the other fails to be PRT[T]-learnable.

\end{proof}

For clarity, we summarize the results of Section \ref{prt-sec} in the following theorem.

\begin{Theorem}
\ 
\begin{enumerate}
\item $PRT \subset PRT[O] \subset PRT[T,O]$,
\item $PRT \subset PRT[T] \subset PRT[T,O]$,
\item $PRT[O] \setminus PRT[T] \neq \emptyset$,
\item $PRT[T] \setminus PRT[O] \neq \emptyset$.
\end{enumerate}
\end{Theorem}

\begin{proof}
All of the claims in the statement follow from Lemmas \ref{prt-t-not-o}, \ref{prt-o-not-t} and \ref{prt-to-not-t-o} and Proposition \ref{prt-t-o}.

\end{proof}

\section{Polynomial Size Dataset}

\begin{Definition}\label{psd-def}
An indexed family, $\mathcal F = \seq{F}{n}$, is \emph{polynomial size dataset learnable} (\emph{PSD-learnable}) if there is a machine $M$ and a polynomial $p$ such that for any enumeration $f$ of $F \in \mathcal F$, $M$ converges to a correct index on an initial segment $f \upto n$ such that $|\{ f(x) : x < n \}| < p(\mi{\mathcal F}{F})$.  If an oracle is accessed, oracle use must also be bounded by $p(\mi{\mathcal F}{F})$.
\end{Definition}

Note that oracle use bounds both the queries to which the oracle reponds in the positive and those to which it responds in the negative.  We shall apply Definition \ref{def-teacher-oracle} to Definition \ref{psd-def} much as we did in the case of Definition \ref{prt-def}.

\begin{Proposition}
$\text{PSD} \subseteq \text{PSD}[O] \subseteq \text{PSD}[T,O]$ and $\text{PSD} \subseteq \text{PSD}[T] \subseteq \text{PSD}[T,O]$.
\end{Proposition}

\begin{proof}
The claim follows from the definitions of PSD, PSD[T], PSD[O] and PSD[T,O].

\end{proof}

Unlike PRT-learning, there are non-trivial PSD-learnable indexed families.

\begin{Example}\label{psd-ex}
Let $\mathcal F$ be an indexed family containing all the finite sets such that $\mi{\mathcal F}{F} = \langle |F|, e \rangle$, where $e$ is the canonical code for $F$.  $\mathcal F$ is PSD-learnable by the learning machine $M$ where $M(a_0 \cdots a_n) = \langle |\content{a_0 \cdots a_n}|,a \rangle$, where $a$ is the canonical code $\content{a_0 \cdots a_n}$.
\end{Example}

\begin{Example}\label{psd-t-o-not-ex}
Let $F_n = [0, 2^n]$.  $\mathcal F = \seq{F}{n}$ is PSD[O]-learnable by a learning machine that uses the exponential query search algorithm to find the greatest element.  $\mathcal F$ is PSD[T]-learnable by the pair $(M,T)$ where $M(a_0^{k_0}, \ldots, a_{n-1}^{k_{n-1}}) = k_0$ and $T(a_0 \cdots a_{k+1}) = a_0^n a_1 \ldots, a_{k+1}$ if $2^n \leq \max\{a_0, \ldots, a_{k+1}\} < 2^{n+1}$ and $\max\{a_0, \ldots, a_{k}\} \allowbreak< 2^n$.  $\mathcal F$ is not PSD-learnable as the learner may be forced to receive $2^{n-1}$ distinct elements before converging to a correct hypothesis. 
\end{Example}

\begin{Lemma}\label{psd-t-not-o}
There is an indexed family which is PSD[T]-learnable, but not PSD[O]-learnable.
\end{Lemma}

\begin{proof}
We prove the claim using a strategy similar to that used in the proof of Lemma \ref{prt-t-not-o}.  Following the notation established in the proof, the only differences are that we define $q(\ell)$ to be the maximum number about which $M$ queries the oracle when it receives inputs which are substrings $\langle 0,1,1,0 \rangle \langle 2,1,1,0 \rangle \cdots \langle 2\ell,1,1,0 \rangle$, given the oracle for $\{ \langle 2i,1,1,0 \rangle : i \leq \ell \}$, and that we define $F_n$ to be a member of $\desc_0(n)$ such that
\[
F_n \cap [0, q(p(\langle m, p^*, 1 \rangle))] = \{ \langle 2i,1,1,0 \rangle : i \leq p(\langle m, p^*, 1 \rangle) \}.
\]

Let $\mathcal F = \seq{F}{n}$.  The proof that $\mathcal F$ is PSD[T]-learnable is exactly the same as the proof that $\mathcal{MSD}$ is PRT[T]-learnable.  That $\mathcal F$ is not PSD[O]-learnable follows from the observation that for abitrary $M = \phi_m$, $M$ cannot distinguish between $F_{\langle m, p^*, 0 \rangle}$ and $F_{\langle m, p^*, 1 \rangle}$ on increasing enumerations without receiving more than $p(\langle m, p^*, 1 \rangle)$ elements of an enumeration of the target.

\end{proof}

\begin{Theorem}
\ 
\begin{enumerate}
\item $\text{PSD} \subset \text{PSD[O]} \subset \text{PSD[T,O]}$,
\item $\text{PSD} \subset \text{PSD[T]} \subset \text{PSD}[T,O]$,
\item $\text{PSD[O]} \setminus \text{PSD[T]} \neq \emptyset$,
\item $\text{PSD[T]} \setminus \text{PSD[O]} \neq \emptyset$.
\end{enumerate}
\end{Theorem}

\begin{proof}
By Lemma \ref{psd-t-not-o} and Example \ref{psd-t-o-not-ex}, we need only prove that $\text{PSD[O]} \setminus \text{PSD[T]} \allowbreak \neq \emptyset$ and $\text{PSD[T]} \cup \text{PSD[O]} \subset \text{PSD[T,O]}$.

Observe that the proof of Lemma \ref{not-prt-t} demonstrates that an indexed family meeting the hypotheses of the lemma is not PSD[T]-learnable.  Thus, Lemma \ref{prt-o-not-t} proves that $\mathcal{CSD} \in \text{PSD[O]} \setminus \text{PSD[T]}$.

Following the proof of Lemma \ref{prt-to-not-t-o}, merging the families constructed in Lemmas \ref{psd-t-not-o} and \ref{prt-o-not-t} with suitable modifications produces an indexed family which is PSD[T,O]-learnable, but neither PSD[T]-learnable nor PSD[O]-learnable.

\end{proof}

It follows from the definitions that $\text{PRT} \subseteq \text{PSD}$, $\text{PRT[O]} \subseteq \text{PSD[O]}$, $\text{PRT[T]} \subseteq \text{PSD[T]}$ and $\text{PRT[T,O]} \subseteq \text{PSD[T,O]}$.  With the following theorem, we show that all three of these containments are strict.

\begin{Theorem}
$\text{PSD} \setminus \text{PRT}[T,O] \neq \emptyset$
\end{Theorem}

\begin{proof}
Let $K$ denote the halting problem.  Define $\mathcal F = \{F_0, F_1, \ldots \}$ where 
\begin{itemize}
\item $F_{2i+1} = \{2i\}$ if $i \not\in K$ and $F_{2i+1} = \{2i, 2i+1\}$ if $i \in K$.
\item $F_{2^{2^i}} = \{2i, 2i+1\}$.
\item For all even numbers $2i \neq 2^{2^k}$ for some $k$, $F_{2i} =\emptyset$.
\end{itemize}
That $\mathcal F$ is $\text{PSD}$-learnable is witnessed by the learner $M$ which outputs 6 on the empty string, outputs $2i+1$ on a string with one unique element which is $2i$ and outputs $2^{2^i}$ and any other string, if the string either contains $2i$ or $2i+1$.  On the other hand, suppose that the learner-teacher pair, $(N, T)$, $\text{PRT[T]}$-learns $\mathcal F$ with polynomial bound $p$.  Computing and returning $2^{2^i}$ cannot be done within $p(2i+1)$ computation steps for more than finitely many values of $i$; thus, for all but finitely many $i$, $i \not\in K$ if and only if $(\exists j)(N(T(2i\ (2i+1)^j))=2^{2^i})$.  Since this would imply that $\overline{K}$ is $\Sigma_1^0$, we have arrived at a contradiction and must conclude that $\mathcal F \not\in \text{PRT}[T]$.  Observe that the use of an oracle does not facilitate learning in this case and so we conclude that $\text{PSD} \setminus \text{PRT[T,O]} \neq \emptyset$.

\end{proof}

\section{Polynomial Mind Changes}

\begin{Definition}
An indexed family, $\mathcal F = \seq{F}{n}$, is \emph{polynomial mind-changes learnable} (\emph{PMC-learnable}) if there is a machine $M$ and a polynomial $p$ such that for every enumeration $f$ of $F \in \mathcal F$, the hypothesis stream, $g$, generated by $M$ on $f$ satisfies $|\{ i : g(i) \neq g(i+1) \}| \leq p(\mi{\mathcal F}{F})$ and the only one that appears infinitely many times in $g$ is an index of $F$.  If an oracle is accessed, oracle use must also be bounded by $p(\mi{\mathcal F}{F})$.
\end{Definition}

We begin with an example exhibiting three PMC-learnable indexed families.

\begin{Example}
Let $\mathcal F$ be an indexed family containing all the finite sets such that $\mi{\mathcal F}{F} = \langle |F|, e \rangle$, where $e$ is the canonical code for $F$.  $\mathcal F$ is PMC-learnable as witnessed by the learning machine $M$ such that $M(a_0 \cdots a_k) = \langle |\content{a_0, \ldots, a_k}|, e \rangle$, where $e$ is the canonical code for the finite set of distinct elements in $a_0, \ldots, a_k$.  On any enumeration of a finite set, $F$, $M$ will change its hypothesis at most $|F|$ times.

Let $\mathcal F = \seq{F}{n}$, where $F_n = [0, 2^n]$.  $\mathcal F$ is PMC-learnable.  Define $M$ such that $M(\sigma)$ is a code for $[0, 2^s]$, where $s$ is the least integer greater than or equal to $\log_2(\max(\sigma))$. 

$\mathcal{MSD}$ is PMC-learned by a learning machine that waits until a descriptor has appeared in the enumeration and then outputs the number the descriptor describes.
\end{Example}

\begin{Theorem}\label{pmc-t}
$\text{PMC[T]} = \text{PMC} = \text{PSD[T]}$ and $\text{PMC[T,O]} = \text{PMC[O]}$.
\end{Theorem}

\begin{proof}
Fix an arbitrary indexed family $\mathcal F$.  If $(M, T)$ PMC[T]-learns $\mathcal F$, then $M \circ T$ PMC-learns $\mathcal F$.  Since every PMC-learnable indexed family is also PMC[T]-learnable, $\text{PMC} = \text{PMC[T]}$.  Similarly, $\text{PMC[T,O]} = \text{PMC[O]}$.

Suppose $(M, T)$ PSD[T]-learns $\mathcal F$ and define $M^*$ such that $M^*(a_0 \cdots a_{k+1}) = \allowbreak M \circ T(a_0 \cdots \allowbreak a_{k+1})$ when $T(a_0 \cdots a_{k+1}) \neq T(a_0 \cdots a_k)$ and $M^*(a_0 \cdots a_{k+1}) = \allowbreak M^*(a_0 \cdots a_{k})$, otherwise.  Since $(M, T)$ PSD[T]-learns $\mathcal F$, the number of distinct elements that $T$ outputs before $M\circ T$ converges to a correct hypothesis is polynomially bounded, hence $M^*$ changes hypothesis a polynomially bounded number of times.  Thus, $\mathcal F \in \text{PMC}$

Define functions $f$ and $g$ such that $f(\sigma) = |\sigma|$ and $g(n, x) = x^n$, the string $x$ repeated $n$ times.  Suppose $M$ PMC-learns $\mathcal F$.  Define $T$ such that $T(\sigma) = g(M(\sigma), \min(\sigma))$ if $M(\sigma)$ is different from $M(\tau)$ for all $\tau \prec \sigma$.  $T(\sigma)$ is undefined otherwise.  $(f, T)$ PSD[T]-learns $\mathcal F$ because it converges to a correct hypothesis after reading a polynomially bounded number of outputs from $T$.

\end{proof}

\begin{Theorem}
$\text{PMC} = \text{PMC[T]} \subset \text{PMC[O]} = \text{PMC[T,O]}$
\end{Theorem}

\begin{proof}
The proof of Lemma \ref{not-prt-t} implies that indexed families which meet the hypotheses are not PMC-learnable.  Thus, $\mathcal{CSD}$ is PMC[O]-learnable, but not PMC-learnable.  By Theorem \ref{pmc-t}, $\text{PMC} = \text{PMC[T]}$ and $\text{PMC[T]} = \text{PMC[T,O]}$.  Hence, the desired claims are true.

\end{proof}

\section{Polynomial Size Characteristic Sample}

\begin{Definition}
An indexed family, $\mathcal F$, is \emph{polynomial size characteristic sample learnable} (\emph{PCS-learnable}) if there is a machine $M$, a polynomial $p$ and a family $\mathcal H$ such that for each $F \in \mathcal F$, there is a corresponding $H \in \mathcal H$ such that $|H| < p(\mi{\mathcal F}{F})$ and if $f$ is an enumeration of $F$, then $M$ outputs the same encoding of $F$ on every initial segment of $f$ whose content includes $H$.  If an oracle is accessed, oracle use must also be bounded by $p(\mi{\mathcal F}{F})$.
\end{Definition}

\begin{Theorem}\label{pcs-o}
\ 
\begin{enumerate}
\item $\text{PCS[O]} \setminus \text{PCS[T]} \neq \emptyset$,\label{pcso-minus-pcst}
\item $\text{PCS[T]} \setminus \text{PCS[O]} \neq \emptyset$ and\label{pcst-minus-pcso}
\item $\text{PCS[T,O]} \setminus (\text{PCS[T]} \cup \text{PCS[O]}) \neq \emptyset$.\label{pcsto-minus-pcst-and-pcso}
\end{enumerate}
\end{Theorem}

\begin{proof}
Define $G_0 = \bbn$, $G_{n} = [0, n]$ for $n > 0$, and $\mathcal G = \seq{G}{n}$.  $\mathcal F$ is PCS[O]-learned by $M$, where $M(a_0 \cdots a_n) = 0$ if the answer to a query about $\max \{a_0, \ldots , a_n \} + 1$ is \textsc{true} and is a code for $[0, \max \{a_0, \ldots , a_n \}]$ otherwise.  Since any string, $\sigma$, can either be extended to an enumeration of $G_0 = \bbn$ or to an enumeration of $G_n = [0, n]$ for any $n \geq \max(\content{\sigma})$, no learner-teacher pair can PCS-learn $\mathcal G$.  Thus, we have proved \ref{pcso-minus-pcst}.

Fix $k$ and suppose that $k = \langle n, p^* \rangle$, where $p$ is an increasing polynomial, and let $M$ be the learner coded by $n$.  Define $E_k$ to be a c.e.~subset of $[2^{2k+1} + 1, 2^{2k+2}]$ that satisfies three conditions.
\begin{itemize}
\item $2^{2k+1}+1 \in E_k$.
\item $|E_k| = p(2k+1)+1$.
\item For any enumeration, $f$, of $[2^{2k+1} + 1, 2^{2k+2}]$, if $E_k \subset f\upto i$ for some $i \leq p(2k+1)$, then $M(f\upto j) = 2k$ for $i \leq j \leq p(2k+1)$.
\end{itemize}
If no such set exists, let $E_k = \emptyset$.  If $E_k \neq \emptyset$, we define a set $D_k$ satisfiying the following conditions.
\begin{itemize}
\item $|D_k| = 2p(2k+1)+1$.
\item $E_k \subset D_k \subset [2^{2k+1} + 1, 2^{2k+2}]$.
\item $D_k$ includes the first $p(2k+1)$ members of $[2^{2k+1} + 1, 2^{2k+2}]$ about which $M$ queries the oracle on a fixed uniformly computable enumeration of $E_k$.
\end{itemize}
If $E_k = \emptyset$, then $D_k = \emptyset$.  To prove \ref{pcst-minus-pcso}, define $\mathcal F = \{F_0, F_1, \ldots \}$ where $F_{2k} = [2^{2k+1} + 1, 2^{2k+2}]$ and $F_{2k+1} = D_k \cup \{2^{2k+1} + 1\}$.  Observe that for any oracle learner, $M$, and polynomial, $p$, there is a $k$ such that either 
\begin{itemize}
\item there is an enumeration of $F_{2k+1}$ on which $M$ converges to $2k$ or $M$ makes more than $p(2k+1)$ oracle queries, or 
\item $M$ does not have a characteristic sample for $F_{2k}$ of size at most $p(2k)$. 
\end{itemize}
Thus, $\mathcal F$ is not $\text{PCS[O]}$-learnable.  On the other hand, consider the learner, $M$, and teacher, $T$, defined as follows.  Once a number of the form $2^{2k+1}+1$ appears in the enumeration, $T$ outputs $2^{2k+1}+1$.  Using $k$, $T$ then determines a natural number, $n$, and polynomial, $p$, such that $k = \langle n, p \rangle$.  The teacher outputs no further numbers until the distinct elements of the enumeration exceeds $2p(2k+1) + 1$.  At this point, $T$ outputs $2^{2k+1} + 2$.  Simultaneously, $T$ calculates $D_k$.  If $D_k$ is nonempty, then $T$ outputs the least element of $[2^{2k+1} + 1, 2^{2k+2}] \setminus D_k$ if it appears in the enumeration.  $M$ returns $2k+1$ if $T$ has output only one element and returns $2k$ if $T$ has output two or more elements.  The learner-teacher pair $\text{PCS[T]}$-learns $\mathcal F$, proving \ref{pcst-minus-pcso}.

We prove \ref{pcsto-minus-pcst-and-pcso} by combining the two families defined above into one family: define $\mathcal H = \{G_0 \otimes \{0\}, F_0 \otimes \{1\}, G_1 \otimes \{0\}, F_1 \otimes \{1\}, \ldots\}$.  Were $\mathcal H$ $\text{PCS[O]}$-learnable, that would imply that $\mathcal F$ is $\text{PCS[O]}$-learnable; similarly, if $\mathcal H$ were $\text{PCS[T]}$-learnable then $\mathcal G$ would also be $\text{PCS[T]}$-learnable.  That $\mathcal H$ is $\text{PCS[T,O]}$-learnable is witnessed by a learner-teacher pair (with access to an oracle) that first waits to see whether the enumeration contains elements of the form $\langle n, 0 \rangle$ or $\langle n, 1 \rangle$ and applies the appropriate learning algorithm as defined above.

\end{proof}

%

The final theorem of this paper illustrates some of the relationships between $\text{PCS}$-learning and the other three types of polynomial-bounded learning.

\begin{Theorem}
\ 
\begin{enumerate}
\item $\text{PMC} \setminus \text{PCS} \neq \emptyset$.\label{pmc-not-pcs}
\item $\text{PSD} \subset \text{PCS}$.\label{psd-sub-pcs}
\item $\text{PCS}\setminus \text{PMC} \neq \emptyset$.\label{pcs-minus-pmc}
\end{enumerate}
\end{Theorem}

\begin{proof}
Observe that the family, $\mathcal F$, defined in the proof of Theorem \ref{pcs-o} is also $\text{PMC}$-learnable.  Consider a learner, $M$, which attempts to compute $D_k$ and returns 0 until a number of the form $2^{2k+1} + 1$ appears in the enumeration.  If this is the only number in the enumeration, then $M$ outputs $2k+1$.  $M$ also outputs $2k + 1$ if $M$ succeeds in computing $D_k$ and every element of the enumeration is in $D_k \cup \{2^{2k+1} + 1\}$.  Otherwise, $M$ outputs $2k$.  Since $\mathcal F \not\in \text{PCS}$, we have proved \ref{pmc-not-pcs}.

We prove \ref{psd-sub-pcs} in two parts.  First, suppose that $M$ $\text{PSD}$-learns a family $\mathcal G = \{G_0, G_1, \ldots \}$ with polynomial bound $p$.  Let $C_i$ denote the first at most $p(i)$ elements of $G_i$.  Define $M^*$ such that $M^*(\sigma) = M(\tau)$, where $\tau$ lists the distinct elements of $\sigma$ in increasing order.  Since $M$ must $\text{PSD}$-learn $G_i$ on the increasing enumeration, $C_i$ must be a characteristic sample for $M^*$ on $G_i$.  Thus, $\text{PSD} \subseteq \text{PCS}$.  Now, consider $\mathcal A = \{A_0, A_1, \ldots \}$ where $A_n = \{n\} \oplus \bbn$.  Consider the string $\alpha_k$ consisting of the odd numbers from $1$ to $2k+1$.  For any member of $\mathcal A$, there is an enumeration that begins with $\alpha_k$.  Consequently, $\mathcal A$ is not $\text{PSD}$-learnable.  Conversely, each member of $\mathcal A$ has a characteristic sample of size $1$.  We conclude that $\text{PSD} \subset \text{PCS}$.

Given $n \in \bbn$, there are unique $i_n$ and $k_n$ such that $n = i_n + 2^{k_n}$ and $1 \leq i_n \leq 2^{k_n}$.  Define $\mathcal G = \{ G_0, G_1, \ldots \}$ where $G_{2n} = \{n\}\oplus[0, 2^n]$ and $G_{2n+1} = \{k_n\}\oplus[0, i_n]$.  In order to $\text{PCS}$-learn $\mathcal F$, we define a learner $M$ as follows.  Let $\sigma$ be an arbitrary string of natural numbers.  If $\sigma$ contains no odd numbers or contains no even numbers, define $M(\sigma) = 0$.  Otherwise, let $2n$ be the least even number in $\sigma$ and let $2m+1$ be the greatest odd number.  $M(\sigma) = 2n$ if $m = 2^n$ and $M(\sigma) = 2k + 1$, where $k = m + 2^n$, if $m \neq 2^n$.  Each member of $\mathcal G$ has a characteristic sample of size 2 for $M$, thus, $M$ $\text{PCS}$-learns $\mathcal G$.  Conversely, suppose that $N$ $\text{PMC}$-learns $\mathcal G$.  For each $n, i, a_1, a_2, \ldots, a_{i-1}$ and $k = i + 2^n$, there is an $a_i$ such that $N(2n\  1\  3^{a_1}\  5^{a_2} \ldots (2i-1)^{a_{i-1}}\ (2i+1)^{a_i}) = 2k+1$.  Thus, for any polynomial there is an $n$ such that $p(2n) < 2^n$ and an enumeration of $G_{2n}$ on which $N$ outputs $2^n$ different hypotheses.  We have proved \ref{pcs-minus-pmc}.

\end{proof}

\section{Acknowledgements}

The authors could like to thank the anonymous referee for numerous useful comments and suggestions.

\end{document}